\documentclass[reqno]{amsart}
\usepackage{amsfonts,amsthm,amsmath,amssymb,bm,mathrsfs,comment}
\usepackage[dvipdfmx]{color,graphicx}
\usepackage{ulem,cancel}
\usepackage{extarrows}

\newtheorem{theorem}{Theorem}
\newtheorem{corollary}[theorem]{Corollary}

\newtheorem{lemma}[theorem]{Lemma}
\usepackage{ascmac}
\def\({\left(}
\def\){\right)}
\def\bder{\text{\cancel{$\partial$}}} 
\begin{document}
\title[Hardy-Leray inequality]{Sharp Hardy-Leray inequality for three-dimensional solenoidal fields with axisymmetric swirl}

\author[N. Hamamoto \and F. Takahashi]{Naoki Hamamoto \and Futoshi Takahashi}

\address{Department of Mathematics, Osaka City University \\
3-3-138 Sugimoto, Sumiyoshi-ku, Osaka 558-8585, Japan}
\email{yhjyoe@yahoo.co.jp {\rm (N.Hamamoto)}} 
\email{futoshi@sci.osaka-cu.ac.jp {\rm (F.Takahashi)}}

\date{\today}

\begin{abstract}
In this paper, we prove Hardy-Leray inequality for three-dimensional solenoidal (i.e., divergence-free) fields with the best constant. 
To derive the best constant, we impose the axisymmetric condition only on the swirl components. 
This partially complements the former work by O. Costin and V. Maz'ya \cite{Costin-Mazya}
on the sharp Hardy-Leray inequality for axisymmetric divergence-free fields.   
\end{abstract}

\maketitle

\section{Introduction}

Let $N\ge3$ be an integer and $\gamma\in\mathbb{R}$ be a real number.
In what follows, 
$\mathcal{D}_\gamma(\mathbb{R}^N)^N$ denotes the set of all smooth vector fields
${\bm u}:\mathbb{R}^N \to \mathbb{R}^N$, ${\bm u}({\bm x})=\big(u_1({\bm x}),u_2({\bm x}),\cdots,u_N({\bm x})\big)$ for ${\bm x}=(x_1,x_2,\cdots,x_N)$
with compact support such that ${\bm u}({\bm 0})={\bm 0}$ if $\gamma\le1-\frac{N}{2}$. 
Then, the Hardy-Leray inequality with weight $\gamma\in\mathbb{R}$ is given by
\begin{equation}
\label{Leray}
  	\(\gamma + \tfrac{N}{2}-1 \)^2 \int_{\mathbb{R}^N} \frac{|{\bm u}|^2}{|{\bm x}|^2}|{\bm x}|^{2\gamma}dx \le \int_{\mathbb{R}^N}|\nabla {\bm u}|^2|{\bm x}|^{2\gamma}dx
\end{equation}
for all ${\bm u}\in\mathcal{D}_\gamma(\mathbb{R}^N)^N$, where the constant $\(\gamma + \frac{N}{2}-1 \)^2$ is known to be sharp. 
This was proved for $N=3,\ \gamma=0$ by J. Leray \cite{Leray} along his study on the Navier-Stokes equations, 
as an $N$-dimensional generalization of the one-dimensional inequality by H. Hardy \cite{Hardy}.
In the context of hydrodynamics,
it is an interesting problem whether the value of the optimal constant increases by imposing ${\bm u}$ to be solenoidal, 
i.e., ${\rm div}\hspace{0.1em}{\bm u}=0$.
Costin and Maz'ya \cite{Costin-Mazya} obtained a positive answer in every dimension $N\ge3$, under the additional assumption of axisymmetry. 

Hereafter let us restrict ourselves to the case $N=3$.
Main result in \cite{Costin-Mazya} in the three dimensional case reads as follows:

\begin{theorem}
[O. Costin and V. Maz'ya \cite{Costin-Mazya}]
\label{Costin-Mazya} 
Let $\gamma \in\mathbb{R}$ and let ${\bm u} \in \mathcal{D}_\gamma(\mathbb{R}^3)^3$ be an axisymmetric solenoidal vector field. 
Then
\[
	C_{\gamma} \int_{\mathbb{R}^3} \frac{|{\bm u}|^2}{|{\bm x}|^2} |{\bm x}|^{2\gamma} dx 
	\le \int_{\mathbb{R}^3} |\nabla {\bm u}|^2 |{\bm x}|^{2\gamma} dx
\] 
holds with the sharp constant 
$C_{\gamma} = \begin{cases}
	\( \gamma + \frac{1}{2}\)^2 \frac{4 + \(\gamma - \frac{3}{2} \)^2}{2 + \(\gamma - \frac{3}{2}\)^2}\ , \quad & \text{for} \ \ \gamma \le 1\   \\	
	\( \gamma + \frac{1}{2} \)^2 + 2 \ ,& \text{for}\ \ \gamma > 1\ .	
\end{cases}$
\end{theorem}

It is clear that $C_\gamma>(\gamma+\frac{1}{2})^2$ for $\gamma\ne-\frac{1}{2}$, 
so we see the solenoidal and axisymmetric constraint improves the best constant of the Hardy-Leray inequality. 
However, since the sole assumption of axisymmetry does not change the optimality of the constant $(\gamma+\frac{1}{2})^2$ in \eqref{Leray}$_{N=3}$, 
it is expected that the condition of axisymmetry in Theorem \ref{Costin-Mazya} can be relaxed.
Indeed, we shall show in this paper that Theorem \ref{Costin-Mazya} does hold by imposing only one component of ${\bm u}$ to be axisymmetric. 
To be more precise, let us introduce the spherical polar coordinates $(\rho,\theta,\varphi)\in[0,\infty)\times[0,\pi]\times[0,2\pi)$ 
in which 
${\bm x}\in\mathbb{R}^3$ is represented by
\[
	{\bm x}=\rho {\bm \sigma}\ ,\quad {\bm \sigma}=(\cos\theta,\sin\theta\cos\varphi,\sin\theta\sin\varphi)\in\mathbb{S}^2.
\]
For each $(\theta,\varphi)$, define the orthonormal frame $({\bm \sigma},{\bm e}_\theta,{\bm e}_\varphi)\in SO(3)$ by
\[
	\begin{split}
	   {\bm e}_\theta&=(-\sin\theta,\cos\theta\cos\varphi,\cos\theta\sin\varphi)\ ,\quad
	 \\ {\bm e}_\varphi&=(0,-\sin\varphi,\cos\varphi)\ .
	\end{split} 
\]
Then a vector field ${\bm u}:\mathbb{R}^3\to\mathbb{R}^3$ at every point ${\bm x}=\rho {\bm \sigma}$ is expanded in that frame as
\[
	{\bm u}={\bm \sigma}u_\rho + {\bm e}_\theta u_\theta+ {\bm e}_\varphi u_\varphi ,
\]
where the last term ${\bm e}_\varphi u_\varphi$ is called the {\it swirl part} of ${\bm u}$, 
which we abbreviate as ${\bm u}_\varphi={\bm e}_\varphi u_\varphi$.
Also, the scalar function $u_\varphi$ is called the {\it swirl component} of ${\bm u}$. 
\footnote{In many papers, swirl component is defined in the cylindrical coordinates in $\mathbb{R}^3$.
However, there are no differences between the two definitions.}
A vector field ${\bm u}$ is called {\it axisymmetric} if its three components $u_\rho$, $u_\theta$ and $u_\varphi$ are independent of $\varphi$. 

Now, let us assume that ${\bm u}\in \mathcal{D}_\gamma(\mathbb{R}^3)^3$, and that its swirl part ${\bm u}_\varphi$ is axisymmetric, i.e., $u_\varphi$ is independent of $\varphi$.  
Then, as we shall show later, ${\bm u}_\varphi$ becomes a solenoidal field and satisfies the inequality
\begin{equation}
\label{HL_swirl}
\((\gamma+\tfrac{1}{2})^2+2\)\int_{\mathbb{R}^3}\frac{|{\bm u}_\varphi|^2}{|{\bm x}|^2}|{\bm x}|^{2\gamma}dx\le\int_{\mathbb{R}^3}|\nabla {\bm u}_\varphi|^2|{\bm x}|^{2\gamma}dx
\end{equation}
with the optimal constant $(\gamma+\tfrac{1}{2})^2+2$ . Since it is just the same as $C_\gamma$ in Theorem \ref{Costin-Mazya} if $\gamma>1$ , 
we observe that the effect of the swirl part is dominant in this case. 
Accordingly, it is also interesting to evaluate the constant if we fix the swirl part of ${\bm u}$. 

Now our main theorem is the following: 
\begin{theorem}
\label{Hamamoto-Takahashi}
Let ${\bm u} \in \mathcal{D}_\gamma(\mathbb{R}^3)^3$ be a solenoidal field.  
If ${\bm u}$ is swirl-free, i.e., ${\bm u}_\varphi\equiv{\bm 0}$, then the inequality
\begin{equation}
\label{HL_ineq}
	C \int_{\mathbb{R}^3} \frac{|{\bm u}|^2}{|{\bm x}|^2} |{\bm x}|^{2\gamma} dx \le \int_{\mathbb{R}^3} |\nabla {\bm u}|^2 |{\bm x}|^{2\gamma} dx
\end{equation}
holds with the optimal constant $C=C_{\gamma,0}$ given by
\[
  \begin{split}
   C_{\gamma,0}&=(\gamma+\tfrac{1}{2})^2+2+\min_{x\ge0}\(x+\tfrac{8(\gamma-1)}{x+2+(\gamma-\frac{3}{2})^2}\)
 \\&=
   \begin{cases}
    \(2\sqrt{\gamma-1}+\sqrt{2}\ \)^2,&\text{for }\
    \frac{3}{2}\le\gamma
    \le\gamma_0,\vspace{0.5em}
    \\
   \( \gamma + \frac{1}{2}\)^2 \frac{4 + \(\gamma - \frac{3}{2} \)^2}{2 + \(\gamma - \frac{3}{2}\)^2}\ , \quad & \text{otherwise}, \end{cases}
  \end{split}  
\]
where $\gamma_0=\frac{3}{2}+(4+\frac{4\sqrt{31}}{3^{3/2}})^{\frac{1}{3}}-\frac{4}{3\(4+\frac{4\sqrt{31}}{3^{3/2}}\)^{\frac{1}{3}}}\fallingdotseq 2.8646556$. 

 More generally, for a given non-zero scalar function $g:\mathbb{R}^3\to\mathbb{R}^3$ which is  independent of $\varphi$ such that ${\bm g}=g {\bm e}_\varphi\in \mathcal{D}_\gamma(\mathbb{R}^3)^3$, 
 let us define 
\[
 \mathcal{G} := \left\{ {\bm u} \in \mathcal{D}_\gamma(\mathbb{R}^3)^3 \, : \,  {\rm div}\,{\bm u}=0\ ,\ \  {\bm u}_{\varphi} = {\bm g} \right\}.
\]
Then \eqref{HL_ineq} holds for any ${\bm u} \in \mathcal{G}$ with the sharp constant $C=C_{\gamma,g}$, where 
\[
	C_{\gamma,g}=\min\left\{C_{\gamma,0}\ ,\ \dfrac{\int_{\mathbb{R}^3}|\nabla{\bm g}|^2|{\bm x}|^{2\gamma}dx}{\int_{\mathbb{R}^3}|{\bm g}|^2|{\bm x}|^{2\gamma-2}dx}\right\}.
 \]
\end{theorem}

Since $C_{\gamma,0} \ge C_{\gamma}$ and 
$
\dfrac{\int_{\mathbb{R}^3}|\nabla{\bm g}|^2|{\bm x}|^{2\gamma}dx}{\int_{\mathbb{R}^3}|{\bm g}|^2|{\bm x}|^{2\gamma-2}dx} \ge \( \gamma + \tfrac{1}{2}\)^2 + 2 \ge C_{\gamma}
$ by \eqref{HL_swirl}
,
we have $C_{\gamma,g}\ge C_\gamma$ for all $\gamma\in\mathbb{R}$. Then, it directly follows from Theorem \ref{Hamamoto-Takahashi} that:
\begin{corollary}
\label{Cor:HT}
Let ${\bm u} \in \mathcal{D}_\gamma(\mathbb{R}^3)^3$ be a solenoidal vector field. 
We assume that ${\bm u}_{\varphi}$ is axisymmetric, i.e., the swirl component $u_{\varphi}$ is independent of $\varphi$. 
Then the inequality
\[
	C_{\gamma} \int_{\mathbb{R}^3} \frac{|{\bm u}|^2}{|{\bm x}|^2} |{\bm x}|^{2\gamma} dx \le \int_{\mathbb{R}^3} |\nabla {\bm u}|^2 |{\bm x}|^{2\gamma} dx
\]
holds with the same constant $C_{\gamma}$ in Theorem \ref{Costin-Mazya}.
\end{corollary}

This corollary shows that the axisymmetry assumption of ${\bm u}$ in Theorem \ref{Costin-Mazya} can be weakened to that of the swirl part ${\bm u}_\varphi$. 
In other words, the non-swirl part ${\bm u}-{\bm u}_\varphi={\bm e}_\rho u_\rho+{\bm e}_\theta u_\theta$ need not be axisymmetric 
to obtain the optimality of the constant in Theorem \ref{Costin-Mazya}.

In the context of fluid mechanics, the effect of the swirl component of a velocity vector field is well-studied from various view points;
see for example, \cite{Lad}, \cite{UI}, \cite{Abe-Seregin}, \cite{KPR}. \cite{KNSS}, \cite{LZ}, \cite{ZZ}, to name a few.
By Corollary \ref{Cor:HT}, we see that the effect of the swirl component is significant also from the view point of general optimal inequalities, 
such as Hardy-Leray inequality with the improved best constant.

\section{Preparation}

Here we give some basic tools that will be needed for the proof of our main theorem. 
As introduced in \S 1, the position of every point ${\bm x}=(x_1,x_2,x_3)\in\mathbb{R}^3$ is written in the spherical polar coordinates 
$(\rho,\theta,\varphi)\in[0,\infty)\times[0,\pi]\times[0,2\pi)$ by ${\bm x}=\rho {\bm \sigma}$, 
where ${\bm \sigma}\in\mathbb{S}^2$ together with the orthonormal basis $({\bm \sigma},{\bm e}_\theta,{\bm e}_\varphi)\in SO(3)$ given by
\begin{equation}
\label{ONB}
	\begin{cases}
	 &{\bm \sigma}= 
	 (\cos\theta,\ \sin\theta\cos\varphi ,\ \sin\theta\sin\varphi) ,\\
	 &{\bm e}_\theta = \partial_\theta {\bm \sigma}=\(-\sin\theta, \cos\theta\cos\varphi, \cos\theta\sin\varphi\),\vspace{0.2em} \\
	&{\bm e}_{\varphi}  =\bder_{\varphi} {\bm \sigma} = (0,-\sin\varphi,\cos\varphi).
	\end{cases}
\end{equation}
Hereafter we use the notations $\partial_\theta=\frac{\partial}{\partial\theta}$, $\partial_\varphi=\frac{\partial}{\partial\varphi}$, 
and also we use the abbreviation 
\[
	\bder_\varphi=\tfrac{1}{\sin\theta}\partial_\varphi.
\]
By differentiating ${\bm \sigma}$, ${\bm e}_\theta$ and ${\bm e}_\varphi$, we verify that

\begin{equation}
\label{DONB}
	\begin{cases}
	\ \partial_\theta {\bm e}_\theta =-{\bm \sigma}\ , 
	\quad \partial_\theta {\bm e}_\varphi= {\bm 0}\ , \\ 
	\ \bder_\varphi {\bm e}_\theta={\bm e}_\varphi\cot\theta\ ,\quad \bder_\varphi {\bm e}_\varphi=-{\bm \sigma}-{\bm e}_\theta\cot\theta\ .
	\end{cases}
\end{equation}
We expand the gradient operator $\nabla=(\frac{\partial}{\partial x_1},\frac{\partial}{\partial x_2},\frac{\partial}{\partial x_3})$ in the frame \eqref{ONB}. 
By use of the chain rule together with \eqref{ONB}, we have
\[
	\partial_\rho 
	=\frac{\partial {\bm x}}{\partial\rho}\cdot\nabla={\bm \sigma} \cdot \nabla\ , \quad
	\partial_{\theta} 
	=\frac{\partial {\bm x}}{\partial\theta}\cdot\nabla=\rho\hspace{0.1em}{\bm e}_{\theta}\cdot \nabla \ ,
	\quad \partial_\varphi=\frac{\partial {\bm x}}{\partial\varphi}\cdot\nabla=(\rho\sin\theta) {\bm e}_\varphi\cdot\nabla \ ,
\]
where ``$\,\cdot\,$'' denotes the standard inner product in $\mathbb{R}^3$. 
Then it turns out that
\begin{equation}
\label{nabla}
	\begin{array}{rl}
	&\nabla =  {\bm \sigma} \partial_\rho + \frac{1\,}{\rho\,}\nabla_{\sigma}\ ,
	 \vspace{0.4em}
	\\ \text{where}
	&\nabla_{\sigma}={\bm e}_\theta\partial_\theta+{\bm e}_\varphi\bder_\varphi
	\ \text{\quad is the  spherical gradient operator.}
	\end{array}
\end{equation}
Now let ${\bm u}={\bm \sigma}u_\rho + {\bm e}_\theta u_\theta+ {\bm e}_\varphi u_\varphi$ be a smooth vector field in $\mathbb{R}^3$. 
By using \eqref{ONB}, \eqref{DONB} and \eqref{nabla} we can check that the divergence of ${\bm u}$ is given by
\begin{align}
\label{divergence}
	\rho\,{\rm div}\,{\bm u}&=\rho\nabla\cdot {\bm u}
	=\big({\bm \sigma}\rho\partial_\rho+ {\bm e}_\theta\partial_\theta+{\bm e}_\varphi\bder_\varphi\big)\cdot({\bm \sigma}u_\rho+{\bm e}_\theta u_\theta+{\bm e}_\varphi u_\varphi) \notag \\
	 &=(\rho\partial_\rho+2)u_\rho+D_\theta u_\theta+\bder_\varphi u_\varphi\ ,
\end{align}
where we have introduced the derivative operator $D_\theta=\partial_\theta+\cot\theta$
which is the $L^2(\mathbb{S}^2)$-adjoint of $-\partial_\theta$:
\begin{equation}
\label{IBP}
- \int_{\mathbb{S}^2}(\partial_\theta f)g\hspace{0.1em}d\sigma=\int_{\mathbb{S}^2}fD_\theta g\hspace{0.1em}d\sigma\ ,\quad d\sigma=\sin\theta d\theta d\varphi\
\end{equation}
for any $f,g\in C^\infty(\mathbb{S}^2)$. 
From \eqref{divergence}, it is clear that
\begin{equation}
\label{div_free}
\left.
\begin{array}{l}
 {\rm div}\hspace{0.1em} {\bm u}_\varphi=0\ ,\vspace{0.25em}\\ {\rm div}\hspace{0.1em}{\bm u}={\rm div}({\bm u}-{\bm u}_\varphi)
 \end{array}\right\}
\quad\
\text{if we assume that}\quad \partial_\varphi u_\varphi=0\ .
\end{equation}
As for the $L^2$ integral of $\nabla {\bm u}$ under such assumption, we have the following lemma.
\begin{lemma}
\label{L2_nabla}
Let ${\bm u}={\bm \sigma}u_\rho+{\bm e}_\theta u_\theta +{\bm e}_\varphi u_\varphi$ be a smooth vector field in $\mathbb{R}^3\backslash\{{\bm 0}\}$. 
Assume that the swirl part ${\bm u}_\varphi={\bm e}_\varphi u_\varphi$ is axisymmetric. 
Then we have 
\[\int_{\mathbb{S}^2}|\nabla {\bm u}|^2d\sigma=\int_{\mathbb{S}^2}|\nabla({\bm u}-{\bm u}_\varphi)|^2d\sigma+\int_{\mathbb{S}^2}|\nabla {\bm u}_\varphi|^2d\sigma\ ,\]
where the two terms in the right-hand side are expressed in terms of components as
\[
\begin{split}
 &\rho^2  \int_{\mathbb{S}^2}|\nabla ({\bm u}-{\bm u}_\varphi)|^2d\sigma=\int_{\mathbb{S}^2}
 \(\begin{array}{l}
	(\rho\partial_\rho u_\rho)^2+(\rho\partial_\rho u_\theta)^2+2u_\rho^2+(\partial_\theta u_\rho)^2 \vspace{0.4em} \\
	+(D_\theta u_\theta)^2-4u_\theta\partial_\theta u_\rho+(\bder_\varphi u_\rho)^2+(\bder_\varphi u_\theta)^2
 \end{array}\)d\sigma \ ,\\
 &\rho^2 \int_{\mathbb{S}^2}|\nabla {\bm u}_\varphi|^2d\sigma=\int_{\mathbb{S}^2}\Big((\rho\partial_\rho u_\varphi)^2+(D_\theta u_\varphi)^2\Big)d\sigma\ .
\end{split} 
\]
\end{lemma}
\begin{proof}
By use of \eqref{nabla}, \eqref{ONB} and \eqref{DONB}, we directly have the following calculations:
 \[
 \begin{split}
 \rho^2 |\nabla {\bm u}|^2&=|\rho\partial_\rho {\bm u}|^2+|\nabla_\sigma {\bm u}|^2,
 \\
  |\nabla_\sigma {\bm u}|^2&=|\partial_\theta {\bm u}|^2+|\bder_\varphi {\bm u}|^2
 \\&=\big|\partial_\theta({\bm \sigma}u_\rho+{\bm e}_\theta u_\theta+{\bm e}_\varphi u_\varphi)\big|^2+\big|\bder_\varphi({\bm \sigma}u_\rho+{\bm e}_\theta u_\theta +{\bm e}_\varphi u_\varphi)\big|^2
  \\&=\big|{\bm e}_\theta u_\rho+{\bm \sigma}\partial_\theta u_\rho-{\bm \sigma}u_\theta+{\bm e}_\theta\partial_\theta u_\theta+{\bm e}_\varphi\partial_\theta u_\varphi\big|^2
  \\&\qquad +\big|{\bm e}_\varphi u_\rho+{\bm \sigma}\bder_\varphi u_\rho+{\bm e}_\varphi(\cot\theta)u_\theta+{\bm e}_\theta\bder_\varphi u_\theta+(-{\bm \sigma}-{\bm e}_\theta\cot\theta)u_\varphi+{\bm e}_\varphi\bder_\varphi u_\varphi\big|^2
  \\&=(\partial_\theta u_\rho-u_\theta)^2+(u_\rho+\partial_\theta u_\theta)^2+(\partial_\theta u_\varphi)^2+(\bder_\varphi u_\rho-u_\varphi)^2\\&\qquad+(\bder_\varphi u_\theta-u_\varphi\cot\theta)^2+(u_\rho+u_\theta\cot\theta+\bder_\varphi u_\varphi)^2.
 \end{split}
 \]
We now assume that ${\bm u}_\varphi$ is axisymmetric. 
Then $\partial_\varphi u_\varphi=0$ and integration by parts yield
 \[
 \begin{split}
	&\int_{\mathbb{S}^2}|\nabla_\sigma {\bm u}|^2d\sigma
	=\int_{\mathbb{S}^2}\Big((\partial_\theta u_\rho-u_\theta)^2+(u_\rho+\partial_\theta u_\theta)^2
	+(\partial_\theta u_\varphi)^2+(\bder_\varphi u_\rho-u_\varphi)^2
	\\&\hspace{10em}+(\bder_\varphi u_\theta-u_\varphi\cot\theta)^2+(u_\rho+u_\theta\cot\theta)^2\Big) d\sigma
  \\&=\int_{\mathbb{S}^2}\left(
  \begin{array}{l}
2 u_\rho^2+(\partial_\theta u_\rho)^2+(\partial_\theta u_\theta)^2+\dfrac{u_\theta^2}{(\sin\theta)^2}\vspace{0.5em}
  \\ -2u_\theta \partial_\theta u_\rho+2 u_\rho D_\theta u_\theta +(\bder_\varphi u_\rho)^2+(\bder_\varphi u_\theta)^2
  \end{array}
\right)d\sigma
+  \int_{\mathbb{S}^2}\bigg( (\partial_\theta u_\varphi)^2 +\frac{u_\varphi^2}{(\sin\theta)^2}\bigg)d\sigma 
	\\&=\int_{\mathbb{S}^2}\Big(2u_\rho^2+(\partial_\theta u_\rho)^2+(D_\theta u_\theta)^2-4u_\theta\partial_\theta u_\rho
	+(\bder_\varphi u_\rho)^2+(\bder_\varphi u_\theta)^2\Big)d\sigma+\int_{\mathbb{S}^2}(D_\theta u_\varphi)^2d\sigma\ ,
 \end{split}
 \]
where the last equality follows from \eqref{IBP} and the commutation relation 
\[\partial_\theta D_\theta-D_\theta\partial_\theta=\frac{-1}{(\sin\theta)^2}\ .\]
Then, adding to the both sides of the above integral equation by
\[
	\int_{\mathbb{S}^2}|\rho\partial_\rho {\bm u}|^2d\sigma
	=\int_{\mathbb{S}^2}\((\rho\partial_\rho u_\rho)^2+(\rho\partial_\rho u_\theta)^2\)d\sigma
	+\int_{\mathbb{S}^2}(\rho\partial_\rho u_\varphi)^2d\sigma\ ,
\]
we have
\[
\begin{split}
 \rho^2
 \int_{\mathbb{S}^2}|\nabla {\bm u}|^2d\sigma&=\int_{\mathbb{S}^2}
 \left(
 \begin{array}{l}
 (\rho\partial_\rho u_\rho)^2+(\rho\partial_\rho u_\theta)^2+2u_\rho^2+(\partial_\theta u_\rho)^2\vspace{0.5em}
 \\
 +(D_\theta u_\theta)^2-4u_\theta\partial_\theta u_\rho
	+(\bder_\varphi u_\rho)^2+(\bder_\varphi u_\theta)^2 
 \end{array}\right)d\sigma
 \\& \quad
 +\int_{\mathbb{S}^2}\left((\rho\partial_\rho u_\varphi)^2+(D_\theta u_\varphi)^2\right)d\sigma\ .
\end{split} 
\] 
This, together with letting $u_\varphi=0$ or $u_\rho=u_\theta=0$, gives the desired formula.
\end{proof}

\section{Proof of Theorem \ref{Hamamoto-Takahashi}}
\label{Preparation}
As in \cite{Costin-Mazya}, let ${\bm u}\not\equiv {\bm 0}$ and let the right-hand side of \eqref{HL_ineq} be finite:
$\int_{\mathbb{R}^3}|\nabla {\bm u}|^2|{\bm x}|^{2\gamma}dx<\infty$, 
since otherwise there is nothing to prove. 
Then the smoothness of ${\bm u}$ implies the existence of an integer $m>-\gamma-\frac{3}{2}$ such that $\nabla {\bm u}=O(|{\bm x}|^m)$ as $|{\bm x}|\to 0$.
Moreover, the condition ${\bm u}({\bm 0})={\bm 0}$ \ for \ $\gamma\le-\frac{1}{2}$ in the definition of $\mathcal{D}_\gamma(\mathbb{R}^3)^3$ leads to
\begin{gather}
\label{origin}
 |{\bm x}|^{\gamma+\frac{1}{2}}{\bm u}({\bm x})=O(|{\bm x}|^\beta)\ ,\qquad
\text{where }\quad \beta=\left\{
\begin{array}{cl} m+\gamma+\frac{3}{2},& {\rm if}\ \ \gamma\le-\frac{1}{2}\ , \ \vspace{0.2em}\\ \gamma+\frac{1}{2},& {\rm if}\ \ \gamma>-\frac{1}{2}\ .
\end{array}\right.
\end{gather}
Since $\beta>0$, this ensures the finiteness of the left-hand side of \eqref{HL_ineq}: $\int_{\mathbb{R}^3}|{\bm u}|^2|{\bm x}|^{2\gamma-2}dx<\infty$. 
We now introduce the vector field ${\bm v}:\mathbb{R}^3\to\mathbb{R}^3$ as the left-hand side of \eqref{origin}:
\[
    {\bm v}({\bm x})=|{\bm x }|^{\gamma+\frac{1}{2}}{\bm u}({\bm x}),
\]
which is called the Brezis-V\'{a}zquez-Maz'ya transformation \cite{Brezis-Vazquez}, \cite{Mazya}. 
Then the right-hand side of \eqref{HL_ineq} is written in terms of ${\bm v}$ as
\[
\begin{split}
	\int_{\mathbb{R}^3} |\nabla {\bm u}|^2|{\bm x}|^{2\gamma}dx
	&=\iint_{\mathbb{R}_+\times\mathbb{S}^2}\big|\nabla\big(\rho^{-(\gamma+\frac{1}{2})}{\bm v}\big)\big|^2\rho^{2\gamma}\rho^2d\rho d\sigma,
	 \\&=\iint_{\mathbb{R}_+\times\mathbb{S}^2}\Big|\rho^{-(\gamma+\frac{1}{2})}\Big(\nabla {\bm v}-(\gamma+\tfrac{1}{2})\tfrac{\nabla\rho}{\rho}{\bm v}\Big)\Big|^2\rho^{2\gamma+2} d\rho d\sigma
	 \\&=\iint_{\mathbb{R}_+\times\mathbb{S}^2}\Big(|\nabla {\bm v}|^2+(\gamma+\tfrac{1}{2})^2 \tfrac{|{\bm v}|^2}{\rho^2}\Big)\rho d\rho d\sigma-(\gamma+\tfrac{1}{2})\iint_{\mathbb{R}_+\times\mathbb{S}^2}\partial_\rho|{\bm v}|^2d\rho
	 \\&=\int_{\mathbb{R}^3}\frac{|\nabla {\bm v}|^2}{|{\bm x}|}dx+(\gamma+\tfrac{1}{2})^2\int_{\mathbb{R}^3}\frac{|{\bm v}|^2}{|{\bm x}|^3}dx\ , 
\end{split} 
\]
where the last equality follows from $|{\bm v}({\bm 0})|=0$ and the support compactness of ${\bm v}$. 
Dividing the both sides by $\int_{\mathbb{R}^3}|{\bm u}|^2|{\bm x}|^{2\gamma-2}dx=\int_{\mathbb{R}^3}|{\bm v}|^2|{\bm x}|^{-3}dx$, 
we have
\begin{equation}
\label{quotient_BV}
\begin{split}
    \frac{\int_{\mathbb{R}^3}|\nabla {\bm u}|^2|{\bm x}|^{2\gamma}dx}{\int_{\mathbb{R}^3}|{\bm u}|^2|{\bm x}|^{2\gamma-2}dx}
	= (\gamma+\tfrac{1}{2})^2+\frac{\int_{\mathbb{R}^3}|\nabla {\bm v}|^2 \frac{dx}{|{\bm x}|}}{\int_{\mathbb{R}^3}|{\bm v}|^2\frac{dx}{|{\bm x}|^3}} .
\end{split}
\end{equation}
Therefore, the minimization problem of the left-hand side, the {\it Hardy-Leray quotient} for ${\bm u}$ with weight $\gamma$, is reduced to that for ${\bm v}$ with weight $-1/2$ .

\subsection{The case ${\bm u}_\varphi\equiv{\bm 0}$}
\label{swirl_free}
In this case, ${\bm u} - {\bm u}_\varphi \not\equiv {\bm 0}$ by the assumption ${\bm u} \not\equiv {\bm 0}$.
Firstly, we evaluate the infimum value of the Hardy-Leray quotient for ${\bm v}\not\equiv {\bm 0}$ under the assumption of swirl free. 
To do so, let $h$ and $f$ denote the components of the 1-D Fourier transform of 
${\bm v}({\bm x})={\bm v}(\rho {\bm \sigma})={\bm v}(e^t {\bm \sigma})$ with respect to the radial variable $t=\log\rho$ :
\[  
	\widehat{\bm v}(\lambda,{\bm \sigma}) =\frac{1}{\sqrt{2\pi}}\int_{\mathbb{R}}e^{-i\lambda t}{\bm v}(e^{t}{\bm \sigma})dt
	={\bm \sigma} h(\lambda,{\bm \sigma})+{\bm e}_\theta f(\lambda,{\bm \sigma})
\]
for $(\lambda,{\bm \sigma})\in\mathbb{R}\times\mathbb{S}^{2}$. 
Then the radial and spherical components of $\widehat{\rho\nabla {\bm v}}$ are given by
\[
	 \widehat{\rho\partial_\rho {\bm v}}(\lambda, {\bm \sigma})=\widehat{\partial_t {\bm v}}(\lambda, {\bm \sigma})
	=i\lambda \hspace{0.1em}\widehat{\bm v}(\lambda, {\bm \sigma})\ ,\quad
	 \widehat{\nabla_\sigma {\bm v}}=\nabla_\sigma \widehat{\bm v}\ .
\]
By use of these relations and Lemma \ref{L2_nabla}, 
the Hardy-Leray quotient of ${\bm v}$ with weight $-1/2$ in \eqref{quotient_BV} is calculated as follows :
\begin{align}
	 \frac{\int_{\mathbb{R}^3}|\nabla {\bm v}|^2\frac{dx}{|{\bm x}|}}{\int_{\mathbb{R}^3}|{\bm v}|^2 \frac{dx}{|{\bm x}|^3}}  
	 &= \frac{\int_{0}^\infty  \(\rho^2\int_{\mathbb{S}^2}|\nabla {\bm v}|^2d\sigma\)\frac{d\rho}{\rho}}
	{\int_0^\infty\big( \int_{{\mathbb{S}^2}}|{\bm v}|^2d\sigma\big)\frac{d\rho}{\rho}}
 \notag
 \\&  =\frac{\displaystyle\iint_{\mathbb{R}_+\times\mathbb{S}^2}
 \(\begin{array}{l}
  (\rho\partial_\rho v_\rho)^2+(\rho\partial_\rho v_\theta)^2+2v_\rho^2+(\partial_\theta v_\rho)^2 \vspace{0.4em} \\
	+(D_\theta v_\theta)^2-4v_\theta\partial_\theta v_\rho+(\bder_\varphi v_\rho)^2+(\bder_\varphi v_\theta)^2
       \end{array}
 \)\frac{d\rho}{\rho} \hspace{0.1em}d\sigma}
 {\displaystyle\iint_{\mathbb{R}_+\times\mathbb{S}^2}(v_\rho^2+v_\theta^2)\frac{d\rho}{\rho} \hspace{0.1em}d\sigma}
 \notag
 \\&=\frac{\displaystyle\iint_{\mathbb{R}\times\mathbb{S}^2}
 \(\begin{array}{l}
  \lambda^2|h|^2+\lambda^2|f|^2+2|h|^2+|\partial_\theta h|^2 \vspace{0.4em} \\
	+|D_\theta f|^2-4 {\rm Re}(\overline{f}\partial_\theta h)+|\bder_\varphi h|^2+|\bder_\varphi f|^2
       \end{array}
 \)d\lambda \hspace{0.1em}d\sigma}
 {\displaystyle\iint_{\mathbb{R}\times\mathbb{S}^2}\big(|h|^2+|f|^2\big)d\lambda \hspace{0.1em}d\sigma}\ .
 \label{HLquotient_v}
\end{align}
Here the last equality follows from the isometric relation 
$\int_{\mathbb{R}}(v_\rho^2+v_\theta^2)\frac{d\rho}{\rho} = \int_{\mathbb{R}}(\widehat{v_\rho}^2+\widehat{v_\theta}^2) d\lambda$.

On the other hand, we now represent the solenoidal condition ${\rm div}\hspace{0.1em} {\bm u}=0$ in terms of $\bm v$ :
\[
\begin{split}
	0 = \rho\, {\rm div}\,{\bm u}&=\rho\,{\rm div}\,(\rho^{-\gamma-\frac{1}{2}}{\bm v})
	=(-\gamma-\tfrac{1}{2})\rho^{-\gamma-\frac{1}{2}}\nabla\rho\cdot {\bm v}+\rho^{-\gamma-\frac{1}{2}}\rho\,{\rm div}\,{\bm v}
	 \\&=\rho^{-\gamma-\frac{1}{2}}\(\rho\,{\rm div}\,{\bm v}-(\gamma+\tfrac{1}{2})v_\rho\) ,
\end{split}
\]
which is equivalent to
\[
	 \rho\,{\rm div}\,{\bm v}=(\rho\partial_\rho+2)v_\rho+D_\theta v_\theta=(\gamma+\tfrac{1}{2})v_\rho\ 
\]
by \eqref{divergence} and the assumption $v_\varphi = 0$.
That is, 
\begin{equation}
\label{solenoidal_v}
	\(\partial_t-\gamma+\tfrac{3}{2}\)v_\rho=-D_\theta v_\theta\ 
\end{equation}
since $\rho \partial_\rho = \partial_t$.
Integrating both sides of \eqref{solenoidal_v} with the complex measure $e^{-i\lambda t}dt$ over $\mathbb{R}$ , 
we find the equivalent solenoidal condition  written in terms of $\widehat{\bm v}={\bm \sigma}h+{\bm e}_\theta f$ as 
\begin{equation}
\label{solenoidal_vhat}
\begin{split}
	&\(i\lambda-\gamma+\tfrac{3}{2}\)h=-D_\theta f \ , \quad \text{that is,} 
	\\ &h=\frac{D_\theta f}{\gamma-\frac{3}{2}-i\lambda}\quad \ \text{for all}\ \ \lambda\ne 0\ .
\end{split}
\end{equation}
Now let us substitute \eqref{solenoidal_vhat} into \eqref{HLquotient_v}. 
Then, integration by parts in each of the numerator and the denominator in \eqref{HLquotient_v} yields
\begin{align}
\frac{\int_{\mathbb{R}^3}|\nabla {\bm v}|^2 \frac{dx}{|{\bm x}|}}{\int_{\mathbb{R}^3} \frac{|{\bm v}|^2}{|{\bm x}|^3}dx}&=
\frac{\displaystyle\iint_{\mathbb{R}\times\mathbb{S}^2}\(
\begin{array}{l}
\frac{(2+\lambda^2)|D_\theta f|^2+|\partial_\theta D_\theta f|^2}{(\gamma-\frac{3}{2})^2+\lambda^2}-4\hspace{0.1em}{\rm Re} \frac{\overline{f}\hspace{0.1em}\partial_\theta D_\theta f}{\gamma-\frac{3}{2}-i\lambda}+|D_\theta f|^2\vspace{0.5em}\\ +\lambda^2|f|^2 +|\bder_\varphi h|^2+|\bder_\varphi f|^2\end{array}
\)d\lambda\hspace{0.1em} d\sigma}
 {\iint_{\mathbb{R}\times\mathbb{S}^2}\(\frac{|D_\theta f|^2}{(\gamma-\tfrac{3}{2})^2+\lambda^2}+|f|^2\)d\lambda\hspace{0.1em} d\sigma}\notag
 \\& =\frac{\displaystyle\iint_{\mathbb{R}\times\mathbb{S}^2}\(
 \begin{array}{l}
 \overline{f(\lambda,{\bm \sigma})}\bigg(\frac{(\partial_\theta D_\theta)^2}{(\gamma-\frac{3}{2})^2+\lambda^2}-\(\frac{\lambda^2+4\gamma-4}{(\gamma-\frac{3}{2})^2+\lambda^2}+1\)\partial_\theta D_\theta+\lambda^2\bigg)f(\lambda,{\bm \sigma})\vspace{0.5em}\\ +|\bder_\varphi h|^2+|\bder_\varphi f|^2\end{array}
\)d\lambda\hspace{0.1em} d\sigma}
 {\iint_{\mathbb{R}\times\mathbb{S}^2}\overline{f(\lambda,{\bm \sigma})}\(\frac{-\partial_\theta D_\theta}{(\gamma-\frac{3}{2})^2+\lambda^2}+1\)f(\lambda,{\bm \sigma})d\lambda \hspace{0.1em}d\sigma}\notag
\\&=\frac{\iint_{\mathbb{R}\times\mathbb{S}^2}\(\overline{f}Q(\lambda,-T_\theta)f+|\bder_\varphi h|^2+|\bder_\varphi f|^2\)d\lambda\hspace{0.1em} d\sigma}{\iint_{\mathbb{R}\times\mathbb{S}^2}\overline{f}q(\lambda,-T_\theta)f d\lambda\hspace{0.1em} d\sigma}\ ,\label{HLquotient_fgh}
\end{align}
where we have introduced the second-order differential operator 
\[
	T_\theta=\partial_\theta D_\theta \ ,
\]
and where $q(\lambda,-T_\theta)$ and $Q(\lambda,-T_\theta)$ are operators defined by the polynomials in $\alpha$,
\[
\left\{
	\begin{array}{l}
	 q(\lambda,\alpha)=\dfrac{\alpha}{(\gamma-\tfrac{3}{2})^2+\lambda^2}+1\ ,\vspace{0.6em}\\
	 Q(\lambda,\alpha)=\dfrac{\alpha^2}{(\gamma-\frac{3}{2})^2+\lambda^2}+\(\dfrac{\lambda^2+4\gamma-4}{(\gamma-\tfrac{3}{2})^2+\lambda^2}+1\)\alpha
	+\lambda^2,\vspace{-0.5em}
	\end{array}\right.
\]
by putting $\alpha=-T_\theta$.
To evaluate \eqref{HLquotient_fgh}, we expand $f$ by using eigenfunctions $\{ \psi_\nu(\theta)\}_{\nu\in\mathbb{N}}\subset C_0^\infty([0,\pi])$ of $-T_\theta$ as
\[
\begin{split}
f(\lambda,\theta,\varphi)&=\sum_{\nu=1}^\infty f_\nu(\lambda,\varphi) \psi_\nu(\theta)\ \ , 
	\quad\ \text{where}\quad
	\left\{ \begin{array}{l}
	 -T_\theta  \psi_\nu=\alpha_\nu  \psi_\nu\ ,\quad \int_{\mathbb{S}^2}| \psi_\nu|^2d\sigma=1\ ,\vspace{0.2em} \\   
\quad	\alpha_\nu=\nu(\nu+1).
	\end{array}\right.
\end{split}
\]
(See Lemma \ref{spec-T} in Appendix.) 
Discarding the non-negative term $|\bder_\varphi h|^2+|\bder_\varphi f|^2$ in the right-hand side of \eqref{HLquotient_fgh}, 
we then see 
\begin{align}
	\frac{\int_{\mathbb{R}^3}|\nabla {\bm v}|^2 \frac{dx}{|{\bm x}|}}{\int_{\mathbb{R}^3}|{\bm v}|^2 \frac{dx}{|{\bm x}|^3}}
	&\ge \frac{\iint_{\mathbb{R}\times\mathbb{S}^2}\overline{f}Q(\lambda,-T_\theta)f d\lambda d\sigma}{\iint_{\mathbb{R}\times\mathbb{S}^2}\overline{f}q(\lambda,-T_\theta)f d\lambda d\sigma}\notag
	 \\&=\frac{\sum_{\nu=1}^\infty\int_{\mathbb{R}\backslash\{0\}} Q(\lambda,\alpha_\nu)|f_\nu(\lambda,\varphi)|^2d\lambda d\sigma}{\sum_{\nu=1}^\infty\int_{\mathbb{R}\backslash\{0\}}q(\lambda,\alpha_\nu)|f_\nu(\lambda,\varphi)|^2d\lambda d\sigma}\notag
	 \\&\ge \inf_{\lambda\ne0}\inf_{\nu\in\mathbb{N}}\frac{Q(\lambda,\alpha_\nu)}{q(\lambda,\alpha_\nu)}
	 =\inf_{x\ge0}\inf_{\nu\in\mathbb{N}}F_\gamma(x,\alpha_\nu),
\label{v_est}
\end{align}
where $F_\gamma$ is defined by
\[
\begin{split}
	F_\gamma(x,\alpha_\nu)
	&=\frac{Q(\sqrt{x},\alpha_\nu)}{q(\sqrt{x},\alpha_\nu)}=\frac{\dfrac{\alpha_\nu^2}{(\gamma-\frac{3}{2})^2+{x}}
	+\(\dfrac{{x}+4\gamma-4}{(\gamma-\tfrac{3}{2})^2+{x}}+1\)\alpha_\nu+{x}}{\dfrac{\alpha_\nu}{(\gamma-\tfrac{3}{2})^2+{x}}+1}
	\\& =x+\(1-\frac{4(1-\gamma)}{x+\alpha_\nu+(\gamma-\frac{3}{2})^2}\)\alpha_\nu
\end{split}
\]
for $x \ge 0$.
Here we note that $F_\gamma(x,\alpha_\nu)$ is just the same as the equation (2.36)$_{n=3}$ in \cite{Costin-Mazya}. 
As in \cite{Costin-Mazya} again, by observing that
\[
\begin{split}
\left.\begin{array}{r}
 F_\gamma(x,\alpha_\nu)\ge F(0,\alpha_\nu)=\(1-\frac{4(1-\gamma)}{\alpha_\nu+(\gamma-\frac{3}{2})^2}\)\alpha_\nu\ ,\vspace{0.7em}
   \\
       \frac{\partial F(0,\alpha_\nu)}{\partial \alpha_\nu}=1-\frac{4(1-\gamma)(\gamma-\frac{3}{2})^2}{\(\alpha_\nu+(\gamma-\frac{3}{2})^2\)^2}>0      
 \end{array}
 \right\}\qquad
 &\text{for}\ \ \gamma\le1\ ,
\\
  \frac{\partial F_\gamma(x,\alpha_\nu)}{\partial \alpha_\nu} =1+\frac{4(\gamma-1)\(x+(\gamma-\frac{3}{2})^2\)}{\(x+\alpha_\nu+(\gamma-\frac{3}{2})^2\)^2}>0
 \qquad& \text{for}\ \ \gamma>1\ ,
\end{split}
\]
we find
\begin{align}
 \inf_{x\ge 0}\inf_{\nu\in\mathbb{N}} F_\gamma(x,\alpha_\nu)&=\min_{x\ge0}F_\gamma(x,\alpha_1)\label{minF}
\\ 
 &=\left\{
 \begin{array}{ll}
    F_\gamma(0,\alpha_1)=2-\frac{8(1-\gamma)}{2+(\gamma-\frac{3}{2})^2}\ ,
				   &\text{for}\ \gamma\le1\ , \vspace{0.5em}\\\displaystyle 2+\min_{x\ge0}\(x+\tfrac{8(\gamma-1)}{x+2+(\gamma-\frac{3}{2})^2}\),
 &\text{for}\ \gamma>1\ .
 \end{array}
 \right.\label{minF_value}
\end{align}
Combining \eqref{minF} to \eqref{v_est}, we arrive at
\[
	\frac{\int_{\mathbb{R}^3}|\nabla{\bm v}|^2 \frac{dx}{|{\bm x}|}}{\int_{\mathbb{R}^3}|{\bm v}|^2 \frac{dx}{|{\bm x}|^3}}
	\ge\min_{x\ge0}F_\gamma(x,\alpha_1).
\]

To show that 
\begin{equation}
 \inf_{\substack{{\bm v}\not\equiv {\bm 0},\\{\rm div}\,{\bm u}=u_\varphi=0}}
  \frac{\int_{\mathbb{R}^3}|\nabla{\bm v}|^2 \frac{dx}{|{\bm x}|}}{\int_{\mathbb{R}^3}|{\bm v}|^2 \frac{dx}{|{\bm x}|^3}}
  =\min_{x\ge0}F_\gamma(x,\alpha_1)\ ,\label{HLquotient_swirlfree}
\end{equation}
let $\lambda_\gamma\in\mathbb{R}$ denote a value of $\lambda$ that attains the minimum of $F_\gamma(\lambda^2,\alpha_1)$. 
Define the sequence $\{{\bm v}_n:\mathbb{R}^3\to\mathbb{R}^3\}_{n\in\mathbb{N}}$ of smooth vector fields by
\[
\begin{split}
	{\bm v}_n({\bm x})&={\bm v}_n(e^t {\bm \sigma})
	=\Big(-{\bm \sigma}D_\theta+{\bm e}_\theta(\partial_t-\gamma+\tfrac{3}{2})\Big)\Big(\xi(\tfrac{t}{n})\cos(\lambda_\gamma t)\sin\theta\Big)
\end{split}
\]
for every $n\in\mathbb{N}$ , where $\xi:\mathbb{R}\to\mathbb{R}$ is an even smooth function $\not\equiv0$ with compact support on $\mathbb{R}$. 
Then, it is clear that ${\bm v}_n$ satisfies \eqref{solenoidal_v},
and so  ${\bm u}_n=\rho^{-\gamma-\frac{1}{2}}{\bm v}_n$ is certainly solenoidal with compact support on $\mathbb{R}^3\backslash\{{\bm 0}\}$. 
Also note that $\psi_1(\theta) = \sin\theta$ is the first eigenfunction of $-T_\theta = -\partial_\theta D_\theta$ associated with $\alpha_1 = 2$, see Lemma \ref{spec-T} in Appendix.
Now let us denote the radial and angular components of $\widehat{{\bm v}_n}$ respectively as $h_n$ and $f_n$. 
Then we see
\[
\begin{split}
	\widehat{\bm v_n}(\lambda,{\bm \sigma})&=\frac{1}{\sqrt{2\pi}}\int_{\mathbb{R}}e^{-i\lambda t}{\bm v_n}(e^t {\bm \sigma})dt
	={\bm \sigma}h_n(\lambda,{\bm \sigma})+{\bm e}_\theta f_n(\lambda,{\bm \sigma})
	\\&=
	 \ \frac{1}{\sqrt{2\pi}} \Big(-2{\bm \sigma}\cos\theta+{\bm e}_\theta\big(i\lambda-\gamma+\tfrac{3}{2}\big)\sin\theta\Big)\int_{\mathbb{R}}e^{-i(\lambda-\lambda_\gamma) t}\xi(\tfrac{t}{n})dt.
\end{split}
\]
This implies $\bder_\varphi h_n=\bder_\varphi f_n=0$ and
\[
	f_n(\lambda,{\bm \sigma})=
	n\big(i\lambda-\gamma+\tfrac{3}{2}\big)(\sin\theta)\,\widehat{\xi}\big(n(\lambda-\lambda_\gamma)\big)\ ,
\]
where \ $\widehat{\xi}(\lambda)=\frac{1}{\sqrt{2\pi}}\int_{\mathbb{R}}e^{-i\lambda t}\xi(t)dt$ \ for all $\lambda\in\mathbb{R}$.
Then, inserting ${\bm v}={\bm v}_n$ or $(h,f)=(h_n,f_n)$ into \eqref{HLquotient_fgh}, we find that
\[
\begin{split}
 \frac{\int_{\mathbb{R}^3}|\nabla{\bm v}_n|^2 \frac{dx}{|{\bm x}|}}{\int_{\mathbb{R}^3}|{\bm v}_n|^2 \frac{dx}{|{\bm x}|^3}}
 &=\frac{\iint_{\mathbb{R}\times\mathbb{S}^2}\overline{f_n}Q(\lambda,-T_\theta)f_nd\lambda\hspace{0.1em} d\sigma}{\iint_{\mathbb{R}\times\mathbb{S}^2}\overline{f_n}q(\lambda,-T_\theta)f_n d\lambda \hspace{0.1em}d\sigma}
 \\&=\frac{\int_{\mathbb{R}}Q(\lambda,\alpha_1)\big(\lambda^2+(\gamma-\frac{3}{2})^2\big)\big|\widehat{\xi}\big(n(\lambda-\lambda_\gamma)\big)\big|^2d\lambda}{\int_{\mathbb{R}}q(\lambda,\alpha_1)\big(\lambda^2+(\gamma-\frac{3}{2})^2\big)\big|\widehat{\xi}\big(n(\lambda-\lambda_\gamma)\big)\big|^2d\lambda}
 \\&= \frac{\int_{\mathbb{R}}Q(\lambda_\gamma+\frac{\lambda}{n},\alpha_1)\big((\frac{\lambda}{n})^2+(\gamma-\frac{3}{2})^2\big)|\widehat{\xi}(\lambda)|^2d\lambda}{\int_{\mathbb{R}}q(\lambda_\gamma+\frac{\lambda}{n},\alpha_1)\big((\frac{\lambda}{n})^2+(\gamma-\frac{3}{2})^2\big)|\widehat{\xi}(\lambda)|^2d\lambda}
 \\&\hspace{-0.7em} \underset{\strut(n\to\infty)}{\longrightarrow}\ \ \lim_{|\lambda|\searrow+0}\frac{Q(\lambda_\gamma+\lambda,\alpha_1)}{q(\lambda_\gamma+\lambda,\alpha_1)}=F_\gamma(\lambda_\gamma^2,\alpha_1)=\min_{\lambda\in\mathbb{R}}F_\gamma(\lambda^2,\alpha_1)\ .
\end{split}
\]
Therefore, $\{{\bm v}_n\}_{n\in\mathbb{N}}$ is certainly a minimizing sequence for the ${\bm v}$-part of the Hardy-Leray quotient, which completes the proof of equation \eqref{HLquotient_swirlfree}.

Returning to \eqref{minF_value}, we compute the minimum value of $F_\gamma(x,\alpha_1)=2+x+\tfrac{8(\gamma-1)}{x+2+(\gamma-\frac{3}{2})^2}$  for $\gamma>1$ :  
by differentiation of this equation, we have
\[
\begin{split}
	 \frac{ \partial}{\partial x} F_\gamma(x,\alpha_1)&=\frac{G_\gamma(x)}{{{\( x+2+\(\gamma-\frac{3}{2}\)^2\) }^{2}}}\ ,
\end{split}
\]
where $G_\gamma(x)=\(x+2+\(\gamma-\frac{3}{2}\)^2\)^2-8(\gamma-1)$ . 
It is easy to check that the quadratic function $G_\gamma$ has the two roots $ x_\gamma^{\pm}=-2-\(\gamma-\frac{3}{2}\)^2\pm 2\sqrt{2}\sqrt{\gamma-1}$ . 
By numerical calculation, they satisfy
\[
\begin{cases}
 x_\gamma^-<0<x_\gamma^+<0,\quad
 &\text{if \ }\tfrac{3}{2}<\gamma<\gamma_0=
 {\textstyle\frac{3}{2}+(4+\frac{4\sqrt{31}}{3^{3/2}})^{\frac{1}{3}}-\frac{4}{3\(4+\frac{4\sqrt{31}}{3^{3/2}}\)^{\frac{1}{3}}}}\ ,
 \\
 x_\gamma^-<x_\gamma^+\le 0,&\text{otherwise.}
\end{cases}
\]
Thus it turns out that 
\[\min_{x\ge0}F_\gamma(x,\alpha_1)=   \begin{cases}
F_\gamma(x_\gamma^+,\alpha_1)=  4\sqrt{2}\sqrt{\gamma-1}-\(\gamma-\frac{3}{2}\)^2,&\text{for }\
    \frac{3}{2}\le\gamma
    \le\gamma_0\ ,\vspace{0.6em}
    \\
 F_\gamma(0,\alpha_1)= 2+  \frac{8(\gamma-1)}{2 + \(\gamma - \frac{3}{2}\)^2}
				      \ , \quad & \text{otherwise}. \end{cases}
\]
Now the computation of \eqref{minF_value} is done. Finally, combining this to \eqref{HLquotient_swirlfree}=\eqref{minF_value} and returning to \eqref{quotient_BV}, we arrive at:
\[
\begin{split}
\inf_{\substack{{\bm u}\not\equiv {\bm 0},\\{\rm div}\,{\bm u}=u_\varphi=0}}\frac{\int_{\mathbb{R}^3}|\nabla {\bm u}|^2|{\bm x}|^{2\gamma}dx}{\int_{\mathbb{R}^3}|{\bm u}|^2|{\bm x}|^{2\gamma-2}dx}
   &=\begin{cases}
(\gamma+\tfrac{1}{2})^2+F_\gamma(x_\gamma^+,\alpha_1)=\(2\sqrt{\gamma-1}+\sqrt{2}\ \)^2
				       ,&\text{for }\
    \frac{3}{2}\le\gamma
    \le\gamma_0\ ,\vspace{0.5em}
    \\
(\gamma+\tfrac{1}{2})^2+ F_\gamma(0,\alpha_1)=  \( \gamma + \frac{1}{2}\)^2 \frac{4 + \(\gamma - \frac{3}{2} \)^2}{2 + \(\gamma - \frac{3}{2}\)^2}
				      \ , \quad & \text{otherwise}, \end{cases}
   \\&=C_{\gamma,0}\ ,
\end{split}
\]
which completes the proof of Theorem \ref{Hamamoto-Takahashi} for ${\bm u}_\varphi\equiv {\bm 0}$.

\subsection{The case ${\bm u}-{\bm u}_\varphi\equiv {\bm 0}$}
\label{swirl_perp_free}

In this case,  ${\bm u}={\bm u}_\varphi \not\equiv {\bm 0}$ is an axisymmetric swirl field by assumption.
This also implies that ${\bm v}-{\bm v}_\varphi\equiv {\bm 0}$, and that ${\bm v}={\bm v}_\varphi \not\equiv {\bm 0}$ is also axisymmetric swirl.
By Lemma \ref{L2_nabla}, the Hardy-Leray quotient for ${\bm v}={\bm v}_\varphi$ with weight $-1/2$ is estimated from below as
\[
\begin{split}
 \frac{\int_{\mathbb{R}^3}|\nabla {\bm v}|^2 \frac{dx}{|{\bm x}|}}{\int_{\mathbb{R}^3}|{\bm v}|^2\frac{dx}{|{\bm x}|^3}}
	&=\frac{\int_{\mathbb{R}^3}|\nabla {\bm v}_\varphi|^2 \frac{dx}{|{\bm x}|}}{\int_{\mathbb{R}^3}|{\bm v}_\varphi|^2\frac{dx}{|{\bm x}|^3}}=\frac{\iint_{\mathbb{R}_+\times\mathbb{S}^2}\((\rho\partial_\rho v_\varphi)^2+(D_\theta v_\varphi)^2 \)\frac{d\rho}{\rho}d\sigma}{\iint_{\mathbb{R}_+\times\mathbb{S}^2}{ v}_\varphi^2\frac{d\rho}{\rho}d\sigma} 
	\\&\ge \frac{\iint_{\mathbb{R}_+\times\mathbb{S}^2}(D_\theta v_\varphi)^2 \frac{d\rho}{\rho}d\sigma}{\iint_{\mathbb{R}_+\times\mathbb{S}^2}{ v}_\varphi^2\frac{d\rho}{\rho}d\sigma} 
	=\frac{\iint_{\mathbb{R}_+\times\mathbb{S}^2}v_\varphi(-T_\theta)v_\varphi \frac{d\rho}{\rho}d\sigma}{\iint_{\mathbb{R}_+\times\mathbb{S}^2}{ v}_\varphi^2\frac{d\rho}{\rho}d\sigma}\ge\alpha_1\ .
\end{split}
\]

To see the infimum of the left-hand side among such ${\bm v}$ is equal to the right-hand side, 
we choose a sequence of axisymmetric swirl fields $\{{\bm v}_n\}_{n\in\mathbb{N}}$  as
\[
	 {\bm v}_n(e^t {\bm \sigma})={\bm e}_\varphi\xi(t/n)\sin\theta\ ,
\]
where  $\xi:\mathbb{R}\to\mathbb{R}$ is a smooth function $\not \equiv 0$ with compact support. 
Then it is easy to check that
\[\frac{\int_{\mathbb{R}^3}|\nabla {\bm v}_n|^2 \frac{dx}{|{\bm x}|}}{\int_{\mathbb{R}^3}|{\bm v}_n|^2\frac{dx}{|{\bm x}|^3}}
	\longrightarrow\alpha_1\quad {\rm as}\ \ n\to\infty\ .
\]
Therefore we have
\[
 \inf_{\substack{{\bm v}={\bm v}_\varphi\not\equiv {\bm 0},\\ \partial_\varphi v_\varphi= 0} }\frac{\int_{\mathbb{R}^3}|\nabla {\bm v}|^2 \frac{dx}{|{\bm x}|}}{\int_{\mathbb{R}^3}|{\bm v}|^2\frac{dx}{|{\bm x}|^3}}=\alpha_1=2\ .
\]
Returning to \eqref{quotient_BV}, we have the inequality \eqref{HL_swirl} for any axisymmetric ${\bm u}_\varphi$ with the optimal constant $(\gamma+\frac{1}{2})^2+2$.

\subsection{The case ${\bm u} \in \mathcal{G}$.}

In this case, the swirl part ${\bm u}_\varphi = {\bm g} = g {\bm e}_{\varphi}\in\mathcal{D}_\gamma(\mathbb{R}^3)^3$ is non-zero and axisymmetric. 
We may assume ${\bm u}-{\bm u}_\varphi\not\equiv {\bm 0}$  by the results in the former subsections.
By Lemma \ref{L2_nabla}, we can split the Hardy-Leray quotient for ${\bm u}$ into the swirl and the non-swirl parts:
\begin{align}
	\frac{\int_{\mathbb{R}^3}|\nabla {\bm u}|^2 |{\bm x}|^{2\gamma}dx}{\int_{\mathbb{R}^3}|{\bm u}|^2|{\bm x}|^{2\gamma-2}dx}
	&=\frac{\int_{\mathbb{R}^3}|\nabla ({\bm u}-{\bm u}_\varphi)|^2 |{\bm x}|^{2\gamma}dx+\int_{\mathbb{R}^3}|\nabla {\bm u}_\varphi|^2 |{\bm x}|^{2\gamma}dx}{\int_{\mathbb{R}^3}|{\bm u}-{\bm u}_\varphi|^2|{\bm x}|^{2\gamma-2}dx+\int_{\mathbb{R}^3}|{\bm u}_\varphi|^2|{\bm x}|^{2\gamma-2}dx}
 \notag
 \\&=\frac{\int_{\mathbb{R}^3}|\nabla ({\bm u}-{\bm u}_\varphi)|^2 |{\bm x}|^{2\gamma}dx+\int_{\mathbb{R}^3}|\nabla {\bm g}|^2 |{\bm x}|^{2\gamma}dx}{\int_{\mathbb{R}^3}|{\bm u}-{\bm u}_\varphi|^2|{\bm x}|^{2\gamma-2}dx+\int_{\mathbb{R}^3}|{\bm g}|^2|{\bm x}|^{2\gamma-2}dx}
 \notag
	 \\&\ge\min\left\{\frac{\int_{\mathbb{R}^3}|\nabla ({\bm u}-{\bm u}_\varphi)|^2 |{\bm x}|^{2\gamma}dx}{\int_{\mathbb{R}^3}|{\bm u}-{\bm u}_\varphi|^2|{\bm x}|^{2\gamma-2}dx}\ ,
	\ \frac{\int_{\mathbb{R}^3}|\nabla {\bm g}|^2 |{\bm x}|^{2\gamma}dx}{\int_{\mathbb{R}^3}|{\bm g}|^2|{\bm x}|^{2\gamma-2}dx}\right\}\notag
	\\& \ge\min\left\{C_{\gamma,0}\ ,
 \ \frac{\int_{\mathbb{R}^3}|\nabla {\bm g}|^2 |{\bm x}|^{2\gamma}dx}{\int_{\mathbb{R}^3}|{\bm g}|^2|{\bm x}|^{2\gamma-2}dx}\right\}=C_{\gamma,g}\ ,
 \label{HLquotient_split}
\end{align}
where the last inequality follows from the result in subsection \ref{swirl_free}, since ${\bm u}-{\bm u}_\varphi$ is swirl-free and solenoidal by \eqref{div_free}. 

To see that the infimum of the left-hand side of \eqref{HLquotient_split} among $\mathcal{G}$ is equal to the right-hand side,
we choose a sequence $\{\tilde{{\bm u}}_n\}_{{n\in\mathbb{N}}}$ of solenoidal and swirl-free fields such that
\[
	\int_{\mathbb{R}^3}|\tilde{{\bm u}}_n|^2|{\bm x}|^{2\gamma-2}dx=1\quad \text{and}\quad
	\int_{\mathbb{R}^3}|\nabla \tilde{{\bm u}}_n|^2 |{\bm x}|^{2\gamma}dx\longrightarrow C_{\gamma,0}\quad\text{as}\ \ n\to\infty.
\]
On the other hand, since $\partial_\varphi g\equiv 0$, the vector field ${\bm g} = g {\bm e}_{\varphi}$ is also solenoidal by \eqref{divergence}.  
Then it follows that the sequence $\{{\bm u}_n=n \tilde{{\bm u}}_n+{\bm g}\}_{n\in\mathbb{N}}$ belongs to $\mathcal{G}$ and that 
$\frac{\int_{\mathbb{R}^3}|\nabla {\bm u}_n|^2 |{\bm x}|^{2\gamma}dx}{\int_{\mathbb{R}^3}|{\bm u}_n|^2|{\bm x}|^{2\gamma-2}dx}\to C_{\gamma,0}$ as $n\to\infty$. 
Consequently, we reach to the desired result.

\section{Appendix}
The second-order derivative operator $T_\theta=\partial_\theta D_\theta=\partial_\theta(\partial_\theta+\cot\theta)$ is self-adjoint in $L^2(\mathbb{S}^2)$. 
Here we specify its spectrum:

\begin{lemma}
\label{spec-T}
Let $C_0^\infty([0,\pi])=\big\{\psi\in C^\infty([0,\pi])\ ;\ \psi(0)=\psi(\pi)=0 \big\}$ and let
 $T=\frac{d}{d{\theta}}(\frac{d}{d\theta}+\cot{\theta})$ be the second-order derivative operator in $C_0^\infty([0,\pi])$. 
Then the set of eigenvalues of $-T$ is given by
\[	
	Spec(-T)=\big\{\alpha_{\nu}=\nu(\nu+1)\ ; \ \nu\in\mathbb{N}\big\}.
\]
Correspondingly, the eigenfunction of $-T$ belonging to $\alpha_\nu$ for every $\nu\in\mathbb{N}$ is given by 
\[
	\psi_\nu(\theta)=P_{\nu-1}(-\cos\theta)\sin\theta\ 
\]
(up to multiplying constant) 
 for some polynomial $P_{\nu-1}$ of degree $\nu-1$. Moreover, the sequence $\{\psi_\nu\}_{\nu\in\mathbb{N}}$ spans a complete orthogonal basis of the Hilbert space $L^2([0,\pi],\sin\theta d\theta)$.
\end{lemma}

\begin{proof}
Let $\psi\in C_0^\infty([0,\pi])$ and put  $\psi(\theta)=\phi(\theta)\sin\theta$. Then we can see the function $\phi:[0,\pi]\to\mathbb{R}$ is smooth in $(0,\pi)$ and continuous on $[0,\pi]$, that is, \[\phi\in C^\infty((0,\pi))\cap C([0,\pi]).\] 
 Also, abbreviating as $\partial_\theta=\frac{d}{d\theta}$, we have
\[\begin{split}
   T\psi&=\partial_\theta\big(\partial_{\theta}+\cot{\theta}\big)(\phi\sin{\theta})
   \\&=(\sin{\theta})\big((\partial_\theta+3\cot\theta)\partial_\theta-2\big)\phi\ .
  \end{split}
\]
 Then the eigenequation $-T\psi=\alpha\psi$ for $\alpha\in\mathbb{R}$ is reduced to
\begin{equation}
\label{eq:-dtheta}
- (\partial_\theta+3\cot\theta)\partial_\theta\phi=(\alpha-2)\phi\ .
\end{equation}
We now transform the variable $\theta$ into $x=-\cos\theta\in[-1,1]$, whose differential obeys the chain rule $\partial_\theta=(\sin\theta)\partial_x$ . Then the derivative operator in the left-hand side of \eqref{eq:-dtheta} is written as
\[
 \begin{split}
  \big(\partial_\theta+3\cot\theta\big)\partial_\theta&=\big(\partial_\theta+3\cot\theta\big)(\sin\theta)\partial_x
  \\&=(\cos\theta)\partial_x+(\sin\theta)\partial_\theta\partial_x+3(\cos\theta)\partial_x
  \\&=(1-x^2)\partial_x^2-4x\partial_x\ .
 \end{split}
\]
Hence equation $\eqref{eq:-dtheta}$ is transformed into
\begin{equation}
 \label{eq:Legendre}
   (1-x^2)\partial_x^2\phi-4x\partial_x\phi+(\alpha-2)\phi=0\ .
\end{equation} 
The solutions of this eigenequation are known to be given by the $5$-dimensional Legendre Polynomials $\{P_{\nu-1}\}_{\nu\in\mathbb{N}}$ : 
\[
 P_{\nu-1}(x)=(1-x^2)^{-1}\Big(\frac{d}{dx}\Big)^{\nu-1}(1-x^2)^{\nu},
\]
 with eigenvalue $\alpha=\alpha_\nu=\nu(\nu+1)$ for each $\nu$. 
 (See, e.g. \cite{Frye}.) 
 Consequently, the $\nu$-th eigenfunction of $-T$ is given by $\psi_\nu(\theta)=P_{\nu-1}(-\cos\theta)\sin\theta$.

 By the Weierstrass approximation theorem, the sequence \[\{\psi_\nu(\theta)=P_{\nu-1}(-\cos\theta)\sin\theta\}_{\nu\in\mathbb{N}}\] spans a dense subspace of $C_0^\infty([0,\pi])$ with respect to the topology of uniform convergence, since every $\psi\in C_0^\infty([0,\pi])$ is expressed as $\psi(\theta)=\phi(\theta)\sin\theta$ for some $\phi\in C([0,\pi])$. Additionally, it is well-known that $C_0^\infty([0,\pi])$ is a dense subspace of $L^2([0,\pi],d\theta)$. Therefore, we obtain \[{\rm span}\{\psi_\nu\}_{\nu\in\mathbb{N}}\underset{\rm dense}{\subset} L^2([0,\pi],d\theta).\]
This holds also with respect to the measure $\sin\theta d\theta$,
 which concludes the lemma.
\end{proof}

\end{document}